\documentclass{article}%
\usepackage{amsmath}
\usepackage{amsfonts}
\usepackage{amssymb}
\usepackage{graphicx}%
\providecommand{\U}[1]{\protect \rule{.1in}{.1in}}
\newtheorem{theorem}{Theorem}

\newtheorem{corollary}[theorem]{Corollary}

\newtheorem{definition}[theorem]{Definition}
\newtheorem{example}[theorem]{Example}

\newtheorem{proposition}[theorem]{Proposition}
\newtheorem{remark}[theorem]{Remark}

\newenvironment{proof}[1][Proof]{\noindent \textbf{#1.} }{\  \rule{0.5em}{0.5em}}
\begin{document}

\begin{center}
Generalized convex functions and their applications in optimality conditions

\bigskip

Mohammad Hossein Alizadeh, Alireza Youhannaee Zanjani

\bigskip

Department of Mathematics, Institute for Advanced Studies in Basic Sciences

(IASBS), 45137-66781 Zanjan, Iran

Email: m.alizadeh@iasbs.ac.ir

Email: alireza.y.z@iasbs.ac.ir

\bigskip

Dedicated to Professor Nicolas Hadjisavvas on the occasion of his 70th
birthday\bigskip
\end{center}

\textbf{Abstract: }We introduce and study the notion of $(e,y)$-conjugate for
a proper and $e$-convex function in locally convex spaces, which is an
extension of the concept of the conjugate. The mutual relationships between
the concepts of $(e,y)$-conjugacy and $e$-subdifferential are presented.
Moreover, some applications of these notions in optimization are established.

\bigskip

\textbf{Keywords:} generalized convex functions, generalized conjugacy,
generalized subdifferentials.

\textbf{2010 Mathematics Subject Classification:} 47H05,49J53, 46B10, 26A27, 26A48.

\section{Introduction}

Several extensions of the convex functions and their relations with their
generalized subdifferentials have been considered in the literature, like
$\epsilon$-convex functions, $\epsilon$-subdifferentials and $\epsilon
$-monotonicity (see \cite{Jo-Luc-MI} and the references therein), and also
$\sigma$-convex functions, $\sigma$-subdifferentials, and $\sigma
$-monotonicity (see \cite{A-javad, MHA-You, Ali 2021, Ali 2018, AHR2012,
Hui-Sun, IKS}). There are many relationships between the convexity of a
function and the monotonicity of its subdifferential. For more details, we
refer to the books of Zalinescu \cite{Zalin. 02} and Clarke \cite{Clarke} and
to the papers \cite{ACL, HLZ, Rock}. Another of these generalizations, which
is the one of interest in this paper, is the notion of $e$-convexity, which
was introduced in \cite{ALi-Mako}, where connections between Hermite--Hadamard
type inequalities and $e$-convexity were established. Furthermore, the notions
of $e$-convexity, $e$-subdifferentials and$\ e$-monotonicity in locally convex
spaces, are studied in \cite{Ali 2022}. The notion of $e$-convexity reduces to
the generalizations that were mentioned earlier here, provided that the
relevant error function is well chosen (see \cite[Remark 1]{ALi-Mako}).

In this stage, we would like to bring to the attention of the readers the
presence of other categories of $E$\emph{-convex and }$e$\emph{-convex
functions} (detailed information on these can be found in \cite{Goberna,
Youness} and the associated literature). We mention that $E$-convexity is an
interesting concept, and underscore its distinction as a separate concept (see
\cite{Youness} and related papers). Also, it is worth noting that, after
publishing \cite{Ali 2022}, Professor Goberna informed the first author by
e-mail about the existence of another class of $e$-convex functions that
already received the same name in abbreviated form (see Chapter 4 of the book
\cite{Goberna}). Of course, their "$e$" is the initial of "evenly" while our
"$e$" is the initial of "error" so both concepts are independent of each
other, but we are afraid that this duplicity could confuse in future to many
readers. To avoid them, we inform the readers about the existence of some
other types of $e$-convex functions.\newline

Without much difficulty, one can observe that the concept of the $e$-convex
function differs from other generalized classes of convex functions, such as
pseudo-convex, and quasi-convex functions. For instance, the function
$f(x)=-x^{2}$ is $e$-convex (see example \ref{Not sigma-conv}, below) with
error function $e(x,y)=(x-y)^{2}$ that is neither pseudo-convex nor
quasi-convex. Indeed, if $f$ satisfied the definition of pseudo-convexity,
then we would expect to have the following:%
\[
f^{\prime}\left(  0\right)  \left(  y-0\right)  \geq0\implies f\left(
y\right)  \geq f\left(  0\right)  \  \  \  \forall y\in%
\mathbb{R}
\text{.}%
\]
Based on the above implication, it can be deduced that for every $y\in%
\mathbb{R}
$ the function $f\left(  y\right)  =-y^{2}\geq0,$ which is a contradiction. To
demonstrate that $f$ is not quasi-convex, consider the values $x=-y=1$ and
$\lambda=\frac{1}{2}$, then the definition of quasi-convexity is not satisfied.

We are uncertain about the relationship between $e$-convex functions and
prox-regular functions (see \cite{P-Roc. 96}): More precisely, the following
question arises:\newline \textbf{Question 1}: \emph{Do }$e$\emph{-convex
functions more general than prox-regular functions?}\newline

The concept of $q$-positivity, which was introduced and explored in
\cite{Gars-Martin-Sim, Mleg}, appears to have a close connection with
$e$-conjugacy. Naturally, the forthcoming question is posed:\newline%
\textbf{Question 2}: \emph{How to relate }$e$\emph{-conjugate to }%
$q$\emph{-positivity?}\newline Indeed we feel to answer the questions $1$, $2$
in detail requires writing another paper.

Besides, another question arises here: Do some well-known results from the
literature regarding conjugacy and subdifferential remain true for
$e$-conjugacy and $e$-subdifferential? In this paper, we give affirmative
answers to this question.\newline Actually, we derive a characterization of
$e$-convex functions, and then we establish a connection between the
$e$-convex function to its Dini-derivatives. Besides, we make a link to
Lipschitz functions. The fact that every global Lipschitz function is also
$e$-convex demonstrates the substantial size of the class of $e$-convex
functions, since Lipschitz functions are very different from convex functions.
Moreover, $e$-conjugate calculus, $e$-subdifferential characterization, and
optimization properties are presented.

We organized the paper as follows. In Section 2, we present some notations and
concepts that are essential for the subsequent sections. In Section 3, various
properties and characterizations of $e$-convex functions are derived. Besides,
it is shown that each Lipschitz function is $e$-convex. The concept of
$e$-conjugacy is introduced and studied in Section 4, where several properties
of this notion and $e$-convex functions are presented. In Section 5, some
applications of the $e$-convexity in optimization are given.

\section{Preliminaries and notations}

Throughout this paper, $X$ and $Y$\ are separated locally convex topological
vector space if not stated explicitly otherwise. We denote the topological
dual space of $X$ by $X^{\ast}$ and duality pairings by $\left \langle
\cdot,\cdot \right \rangle $. We will use the following convention%
\begin{equation}
(+\infty)+(-\infty)=(-\infty)+(+\infty)=-\infty. \label{Convention}%
\end{equation}
For $\Omega \subset X$, we denote by $\operatorname*{bd}\left(  \Omega \right)
$ the boundary points of $\Omega$ and by $\operatorname*{co}\Omega$ the convex
hull of $\Omega$. It is well known that the class $\mathcal{N}_{0}$ of the
closed, convex and balanced neighborhoods of the origin is a base of
neighborhoods of $0$. We use $\mathcal{N}_{x}$\ as a neighborhood base for $x$
(for more detail see \cite{Zalin. 02}).

Let $T$ be a set-valued map\ from $X$ to $X^{\ast}$. The \emph{domain} and
\emph{graph} of $T$ are, respectively, defined by%
\[
D\left(  T\right)  =\left \{  x\in X:T\left(  x\right)  \neq \emptyset \right \}
,
\]%
\[
\operatorname*{gr}T=\left \{  \left(  x,x^{\ast}\right)  \in X\times X^{\ast
}:x\in D\left(  T\right)  ,\text{ and\ }x^{\ast}\in T\left(  x\right)
\right \}  \text{.}%
\]

We recall that $T$ is \emph{monotone} if%
\[
\langle x-y,x^{\ast}-y^{\ast}\rangle \geq0
\]
for all $x,y\in X$ and $x^{\ast}\in T\left(  x\right)  ,y^{\ast}\in T\left(
y\right)  $. A monotone operator is called \emph{maximal monotone} if it has
no monotone extension other than itself.

Let $f:X\rightarrow%
\mathbb{R}
\cup \left \{  +\infty \right \}  $ be an extended real-valued function. Its
\emph{domain} (or effective domain) and \emph{epigraph}, respectively are
defined by $\operatorname*{dom}f:=\left \{  x\in X:f\left(  x\right)
<+\infty \right \}  ,\ $and\ $\operatorname*{epi}f:=\left \{  \left(
x,\lambda \right)  \in X\times%
\mathbb{R}
:f\left(  x\right)  \leq \lambda \right \}  $. The function $f$ is called
\emph{proper} if $\operatorname*{dom}f\neq \emptyset$. Let $f:X\rightarrow%
\mathbb{R}
\cup \left \{  +\infty \right \}  $ be a proper function. The
\emph{subdifferential} of $f$ at $x\in \operatorname*{dom}f$\ (in the sense of
Convex Analysis)\ is defined by%
\begin{equation}
\partial f\left(  x\right)  =\left \{  x^{\ast}\in X^{\ast}:\left \langle
y-x,x^{\ast}\right \rangle \leq f\left(  y\right)  -f\left(  x\right)
\qquad \forall y\in X\right \}  , \label{subdiff-conana}%
\end{equation}
and $\partial f\left(  x\right)  $ is empty if $x$ is not in the domain of
$f$.\newline For a proper function $f:X\rightarrow%
\mathbb{R}
\cup \left \{  +\infty \right \}  $ the \emph{Clarke-Rockafellar generalized
directional derivative} at $x$ in a direction $z\in X$ is defined by%
\[
f^{\uparrow}\left(  x,z\right)  =\sup_{\delta>0}\limsup_{(y,\alpha)\overset
{f}{\rightarrow}x,\lambda \searrow0}\inf_{u\in B\left(  z,\delta \right)  }%
\frac{f\left(  y+\lambda u\right)  -\alpha}{\lambda}%
\]
where $(y,\alpha)\overset{f}{\rightarrow}x$ means that $y\rightarrow
x,\alpha \rightarrow f\left(  x\right)  $ and $\alpha \geq f\left(  y\right)  $.
The Clarke-Rockafellar subdifferential \cite{Clarke, Zalin. 02} of $f$ at
$x\in \operatorname*{dom}f$ is defined by
\[
\partial^{CR}f\left(  x\right)  =\left \{  x^{\ast}\in X^{\ast}:\langle
x^{\ast},z\rangle \leq f^{\uparrow}\left(  x,z\right)  \quad \forall z\in
X\right \}  \text{.}%
\]

We recall that \cite{Ali 2018}, for a given function $f:X\rightarrow%
\mathbb{R}
\cup \left \{  +\infty \right \}  $ and a map $\sigma:X\rightarrow%
\mathbb{R}
_{+}\cup \left \{  +\infty \right \}  $, such that $\operatorname*{dom}%
f\subseteq \operatorname*{dom}\sigma$, we say that $f$ is $\sigma
$\emph{-convex} if%
\[
f\left(  tx+\left(  1-t\right)  y\right)  \leq tf\left(  x\right)  +\left(
1-t\right)  f\left(  y\right)  +t\left(  1-t\right)  \min \{ \sigma
(x),\sigma(y)\}||x-y||
\]
for all $x,y\in X$, and $t\in]0,1[$.

The concept of $\sigma$-convex functions, which extends the $\epsilon$-convex
functions (see \cite{Jo-Luc-MI, NLM 97}), has been explored in \cite{Ali 2021,
Ali 2018, Hui-Sun}.

We recall that a bifunction $e:X\times X\rightarrow%
\mathbb{R}
\cup \left \{  +\infty \right \}  $ is an error function if $\operatorname*{dom}%
e:=\left \{  \left(  x,y\right)  \in X\times X:e\left(  x,y\right)
<+\infty \right \}  \neq \emptyset$, $e$\ is nonnegative, and also it is
symmetric, i.e., $e\left(  x,y\right)  =e\left(  y,x\right)  $. Throughout
this article $e$ will be an error function. The notion of $e$-convexity is
introduced and studied in \cite{ALi-Mako}.

\begin{definition}
\cite{ALi-Mako} Let a function $f:X\rightarrow%
\mathbb{R}
\cup \left \{  +\infty \right \}  $ and an error function $e$ such that
$\operatorname*{dom}f\times \operatorname*{dom}f\subseteq \operatorname*{dom}e$
be given. Then $f$ is called $e$\emph{-convex} if
\begin{equation}
f\left(  tx+\left(  1-t\right)  y\right)  \leq tf\left(  x\right)  +\left(
1-t\right)  f\left(  y\right)  +t\left(  1-t\right)  e\left(  x,y\right)
\label{convexineq}%
\end{equation}
for all $x,y\in X$, and $t\in]0,1[$.
\end{definition}

Consider an operator $T:X\rightarrow2^{X^{\ast}}$ and an error function $e$,
such that $D\left(  T\right)  \times D\left(  T\right)  \subset
\operatorname*{dom}e$. Then $T$ is said to be $e$\emph{-monotone} if for every
$x,y\in D(T),\ x^{\ast}\in T\left(  x\right)  $ and $y^{\ast}\in T\left(
y\right)  $,%
\begin{equation}
\langle x-y,x^{\ast}-y^{\ast}\rangle \geq-e\left(  x,y\right)  . \label{sigma1}%
\end{equation}

In \cite{Ali 2022}, the concept of $e$-subdifferential is presented and
analyzed. We recall it here:

Suppose that $f:X\rightarrow%
\mathbb{R}
\cup \left \{  +\infty \right \}  $ is a proper function. The $e$-subdifferential
of $f$ is the multivalued operator $\partial^{e}f:X\rightarrow2^{X^{\ast}}$
defined by%
\[
\partial^{e}f\left(  x\right)  :=\left \{  x^{\ast}\in X^{\ast}:\langle
z,x^{\ast}\rangle \leq f\left(  x+z\right)  -f\left(  x\right)  +e(x+z,x)\quad
\forall z\in X\right \}
\]
\newline if $x\in \operatorname*{dom}f;$ otherwise it is empty. Equivalently,
\[
\partial^{e}f\left(  x\right)  =\left \{  x^{\ast}\in X^{\ast}:\left \langle
y-x,x^{\ast}\right \rangle \leq f\left(  y\right)  -f\left(  x\right)
+e(x,y)\quad \forall y\in X\right \}  .
\]
For the convenience of the readers regarding Question 1, we recall some
related subjects from \cite{P-Roc. 96}, a function $f:%
\mathbb{R}
^{n}\rightarrow \overline{%
\mathbb{R}
}$ is primal-lower-nice (p.l.n.) at $\overline{x}$, a point where $f$ is
finite, if there exist $R>0,c>0$ and $\epsilon>0$ with the property that%
\[
f\left(  x^{\prime}\right)  >f\left(  x\right)  +\left \langle v,x^{\prime
}-x\right \rangle -\frac{r}{2}\left \Vert x^{\prime}-x\right \Vert ^{2},
\]
whenever $r>R$, $\left \Vert v\right \Vert <cr,v\in \partial f\left(  x\right)
,\left \Vert \overline{x}-x^{\prime}\right \Vert <\epsilon$ with $x^{\prime}\neq
x$ and $\left \Vert x-\overline{x}\right \Vert <\epsilon$.

\section{Some properties of $e$-convex functions}

Some properties and characterizations regarding the $e$-convex functions are
studied \cite{Ali 2022, ALi-Mako}. In the following proposition, we will
present another characterization of $e$-convex functions in locally convex
spaces. Indeed, Huang and Sun in \cite{Hui-Sun} presented two
characterizations of $\sigma$-convex functions, and in the next proposition,
we generalized their result.

\begin{proposition}
Let a function $f:X\rightarrow%
\mathbb{R}
\cup \left \{  +\infty \right \}  $ and an error function $e$ such that
$\operatorname*{dom}f\times \operatorname*{dom}f\subset \operatorname*{dom}e$ be
given. Then the following statements are equivalent:

(i) $f$ is $e$-convex;

(ii) for each $x,y\in X$ and for all $r,s,t\in%
\mathbb{R}
$ such that $x+sy\in \operatorname*{dom}f$%
\begin{equation}
\frac{f\left(  x+sy\right)  -f\left(  x+ry\right)  }{s-r}\leq \frac{f\left(
x+ty\right)  -f\left(  x+sy\right)  }{t-s}+\frac{1}{t-r}e\left(
x+ty,x+ry\right)  \label{char 1}%
\end{equation}

whenever $r<s<t$.

(iii) for every $x,y\in X$ and for all $s,t\in%
\mathbb{R}
$%
\begin{equation}
\frac{f\left(  x+sy\right)  -f\left(  x\right)  }{s}\leq \frac{f\left(
x+ty\right)  -f\left(  x\right)  }{t}+\left(  \frac{t-s}{t^{2}}\right)
e\left(  x+ty,x\right)  \label{char 2}%
\end{equation}

whenever $0<s<t$.
\end{proposition}

\begin{proof}
To prove (i)$\implies$(ii), put $\lambda:=\frac{s-r}{t-r}$. Then $\lambda
\in]0,1[$ and $1-\lambda=\frac{t-s}{t-r}$. In addition, by a simple
calculation, we obtain%
\begin{equation}
x+sy=\lambda(x+ty)+\left(  1-\lambda \right)  \left(  x+ry\right)  \text{.}
\label{calcu}%
\end{equation}
Using the $e$-convexity assumption and (\ref{calcu}) we get%
\begin{align*}
f\left(  x+sy\right)   &  =f\left(  \lambda(x+ty)+\left(  1-\lambda \right)
\left(  x+ry\right)  \right) \\
&  \leq \lambda f(x+ty)+\left(  1-\lambda \right)  f\left(  x+ry\right)
+\lambda \left(  1-\lambda \right)  e\left(  x+ty,x+ry\right)  .
\end{align*}
Replace the values of $\lambda=\frac{s-r}{t-r}$ and $1-\lambda=\frac{t-s}%
{t-r}$ in the above inequality and multiply both sides by $\left(  t-r\right)
$. Hence%
\begin{align*}
\left(  t-s\right)  \left(  f\left(  x+sy\right)  -f\left(  x+ry\right)
\right)   &  \leq \left(  s-r\right)  \left(  f(x+ty)-f\left(  x+sy\right)
\right) \\
&  +\frac{\left(  s-r\right)  \left(  t-s\right)  }{t-r}e\left(
x+ty,x+ry\right)  .
\end{align*}
Multiplying both sides of the above inequality by $\frac{1}{\left(
s-r\right)  \left(  t-s\right)  }$, we infer (\ref{char 1}).\newline To show
(ii)$\implies$(iii), if we take $r=0$ in (\ref{char 1}), then%
\[
\frac{f\left(  x+sy\right)  -f\left(  x\right)  }{s}\leq \frac{f\left(
x+ty\right)  -f\left(  x+sy\right)  }{t-s}+\frac{1}{t}e\left(  x+ty,x\right)
\text{.}%
\]
Now we multiply both sides of the previous inequality by $s\left(  t-s\right)
$, so%
\[
t\left(  f\left(  x+sy\right)  -f\left(  x\right)  \right)  \leq s\left(
f(x+ty)-f\left(  x\right)  \right)  +\frac{s\left(  t-s\right)  }{t}e\left(
x+ty,x\right)  .
\]
Multiply both sides of the above inequality by $\frac{1}{st}$ and then
simplify the resulting inequality we get (\ref{char 2}).\newline To prove
(iii)$\implies$(i), let $s\in]0,1[$. If we put $t=1$ and $y=z-x$ in
(\ref{char 2}), then%
\[
\frac{f\left(  sz+\left(  1-s\right)  x\right)  -f\left(  x\right)  }{s}\leq
f\left(  z\right)  -f\left(  x\right)  +\left(  1-s\right)  e\left(
z,x\right)  \text{.}%
\]
This implies that $f\left(  sz+\left(  1-s\right)  x\right)  \leq sf\left(
z\right)  +\left(  1-s\right)  f\left(  x\right)  +s\left(  1-s\right)
e\left(  x,z\right)  $, i.e., $f$ is $e$-convex.
\end{proof}

In the following, we provide an example of an $e$-convex function such that it
is not $\sigma$-convex for any map $\sigma$.

\begin{example}
\label{Not sigma-conv}Consider the function $f:%
\mathbb{R}
\rightarrow%
\mathbb{R}
$ defined by $f(x)=-x^{2}$. Then $f$ is $e$-convex with error function
$e(x,y)=(x-y)^{2}$ and it is not $\sigma$-convex for any $\sigma:%
\mathbb{R}
\rightarrow%
\mathbb{R}
_{+}$. Note that here, for all $x,y\in%
\mathbb{R}
_{+}$ we have $e(x,y)=e(y,x)\geq0$ and $e(x,x)=0$.\newline First we show the
$e$-convexity of $f$. Let $x,y\in \operatorname*{dom}f=%
\mathbb{R}
$ and $t\in]0,1[$. Then%
\begin{align*}
&  f\left(  tx+\left(  1-t\right)  y\right)  -tf\left(  x\right)  -\left(
1-t\right)  f\left(  y\right) \\
&  =-\left(  tx+\left(  1-t\right)  y\right)  ^{2}+tx^{2}+\left(  1-t\right)
y^{2}\\
&  =t\left[  x^{2}-\left(  tx+\left(  1-t\right)  y\right)  ^{2}\right]
+\left(  1-t\right)  \left[  y^{2}-\left(  tx+\left(  1-t\right)  y\right)
^{2}\right] \\
&  =t\left(  1-t\right)  \left(  x-y\right)  \left[  \left(  1+t\right)
x+\left(  1-t\right)  y-tx+\left(  t-2\right)  y\right] \\
&  =t\left(  1-t\right)  \left(  x-y\right)  \left(  x-y\right)  =t\left(
1-t\right)  \left(  x-y\right)  ^{2}\text{.}%
\end{align*}

Now we prove that function $f(x)=-x^{2}$ is not $\sigma$-convex for any
$\sigma:%
\mathbb{R}
\rightarrow%
\mathbb{R}
_{+}$. Suppose, on the contrary, that there exists map $\sigma:%
\mathbb{R}
\rightarrow%
\mathbb{R}
_{+}$ such that $f$ is $\sigma$-convex. Then by definition for each $x,y\in%
\mathbb{R}
$ and $t\in]0,1[$ we have:%
\begin{equation}
f\left(  tx+\left(  1-t\right)  y\right)  \leq tf\left(  x\right)  +\left(
1-t\right)  f\left(  y\right)  +t\left(  1-t\right)  \overline{\sigma}\left(
x,y\right)  \left \vert x-y\right \vert , \label{1 in examp}%
\end{equation}
where $\overline{\sigma}\left(  x,y\right)  =\min \left \{  \sigma \left(
x\right)  ,\sigma \left(  y\right)  \right \}  $.\newline We can also assume
without loss of generality that $0<y<x$. It follows from (\ref{1 in examp})
that%
\[
\left[  f\left(  tx+\left(  1-t\right)  y\right)  -f\left(  y\right)  \right]
+t\left[  f\left(  y\right)  -f\left(  x\right)  \right]  \leq t\left(
1-t\right)  \overline{\sigma}\left(  x,y\right)  (x-y).
\]
Divide both sides of the above inequality by $t(1-t)(x-y)$ and use the
definition of $f$. Then we get%
\begin{align*}
\sigma \left(  y\right)   &  \geq \overline{\sigma}\left(  x,y\right)  \geq
\frac{-\left(  y+t(x-y)\right)  ^{2}+y^{2}}{t(1-t)(x-y)}+\frac{x^{2}-y^{2}%
}{(1-t)(x-y)}\\
&  =\frac{-t^{2}\left(  x-y\right)  ^{2}-2ty\left(  x-y\right)  }%
{t(1-t)(x-y)}+\frac{x+y}{1-t}=x-y\text{.}%
\end{align*}
Tending $x\rightarrow+\infty$, we get $\sigma \left(  y\right)  \geq+\infty$.
This is a contradiction.
\end{example}

In the subsequent result, we show that the product of any two convex functions
is an $e$-convex function.

\begin{proposition}
Suppose that $f,g:X\rightarrow%
\mathbb{R}
\cup \left \{  +\infty \right \}  $ are convex and nonnegative functions and
$\operatorname*{dom}f\cap \operatorname*{dom}g\neq \emptyset$. Then $f\times g$
is $e$-convex with $e\left(  x,y\right)  =f\left(  x\right)  g\left(
y\right)  +f\left(  y\right)  g\left(  x\right)  $ for all $x,y\in X$.
\end{proposition}

\begin{proof}
Let $x,y\in \operatorname*{dom}f\cap \operatorname*{dom}g$ and $t\in]0,1[$.Then%
\begin{align*}
&  (f\times g)\left(  tx+\left(  1-t\right)  y\right)  -t(f\times g)\left(
x\right)  -\left(  1-t\right)  (f\times g)\left(  y\right) \\
&  =f\left(  tx+\left(  1-t\right)  y\right)  g\left(  tx+\left(  1-t\right)
y\right)  -tf\left(  x\right)  g\left(  x\right)  -\left(  1-t\right)
f\left(  y\right)  g\left(  y\right) \\
&  \leq \left[  tf\left(  x\right)  +\left(  1-t\right)  f\left(  y\right)
\right]  \left[  tg\left(  x\right)  +\left(  1-t\right)  g\left(  y\right)
\right]  -tf\left(  x\right)  g\left(  x\right)  -\left(  1-t\right)  f\left(
y\right)  g\left(  y\right) \\
&  =t^{2}f\left(  x\right)  g\left(  x\right)  +t\left(  1-t\right)  f\left(
x\right)  g\left(  y\right)  +t\left(  1-t\right)  f\left(  y\right)  g\left(
x\right)  +\left(  1-t\right)  ^{2}f\left(  y\right)  g\left(  y\right) \\
&  \  \  \  \  \  \  \  \  \  \  \  \  \  \  \  \  \ -tf\left(  x\right)  g\left(  x\right)
-\left(  1-t\right)  f\left(  y\right)  g\left(  y\right) \\
&  \leq tf\left(  x\right)  g\left(  x\right)  +t\left(  1-t\right)  f\left(
x\right)  g\left(  y\right)  +t\left(  1-t\right)  f\left(  y\right)  g\left(
x\right)  +\left(  1-t\right)  f\left(  y\right)  g\left(  y\right) \\
&  \  \  \  \  \  \  \  \  \  \  \  \  \  \  \ -tf\left(  x\right)  g\left(  x\right)
-\left(  1-t\right)  f\left(  y\right)  g\left(  y\right) \\
&  =t\left(  1-t\right)  \left[  f\left(  x\right)  g\left(  y\right)
+f\left(  y\right)  g\left(  x\right)  \right]
\end{align*}
The proof is completed.
\end{proof}

The following proposition is an extension of the well-known result from the
literature (for instance, see \cite[Theorem 2.1.11]{Zalin. 02}).

\begin{proposition}
Let a function $f:X\rightarrow%
\mathbb{R}
\cup \left \{  +\infty \right \}  $ and an error function $e$ such that
$\operatorname*{dom}f\times \operatorname*{dom}f\subset \operatorname*{dom}e$ be
given. Suppose that $D\subset \operatorname*{dom}f$ is a nonempty, convex and
open set of $X$, and $f|_{D}$ is G\^{a}teaux differentiable. Consider the
following statements

(i) $f$ is $e$-convex;

(ii) for all $x,y\in D:\left \langle \nabla f\left(  x\right)
,y-x\right \rangle \leq f\left(  y\right)  -f\left(  x\right)  +e(x,y)$;

(iii) $\nabla f$ is $2e$-monotone on $D$.

Then (i)$\implies$(ii)$\implies$(iii).
\end{proposition}

\begin{proof}
To show (i)$\implies$(ii), let $x,y\in D$ and $t\in]0,1[$. Then by using the
$e$-convexity assumption, we have%
\begin{align*}
\frac{f\left(  x+t(y-x)\right)  -f\left(  x\right)  }{t}  &  =\frac{f\left(
\left(  1-t\right)  x+ty\right)  -f\left(  x\right)  }{t}\\
&  \leq \frac{t\left(  f\left(  y\right)  -f\left(  x\right)  \right)
+t\left(  1-t\right)  e(x,y)}{t}\\
&  =\left(  f\left(  y\right)  -f\left(  x\right)  \right)  +\left(
1-t\right)  e(x,y).
\end{align*}
Tending $t\rightarrow0^{+}$ we obtain (ii).\newline To prove (ii)$\implies
$(iii), assume that $x,y\in D$. Then by (ii)%
\[
\left \langle \nabla f\left(  x\right)  ,y-x\right \rangle \leq f\left(
y\right)  -f\left(  x\right)  +e(x,y).
\]
By changing the role of $x$ and $y$ in the above inequality, we get%
\[
\left \langle \nabla f\left(  y\right)  ,x-y\right \rangle \leq f\left(
x\right)  -f\left(  y\right)  +e(x,y).
\]
If we add the two preceding inequalities, then obtain (iii).
\end{proof}

Let $f:X\rightarrow \mathbb{R}\cup \{+\infty \}$ be a proper function. As in
\cite{Zalin. 02},\ the upper Dini directional derivative of $f$ at $x\in X$ in
the direction $u\in X$ is defined by%
\[
\overline{D}f\left(  x,u\right)  :=\left \{
\begin{array}
[c]{cc}%
\underset{t\rightarrow0^{+}}{\lim \sup}\frac{f\left(  x+tu\right)  -f\left(
x\right)  }{t} & \text{if \ }x\in \operatorname*{dom}f,\\
-\infty & \text{otherwise.}%
\end{array}
\right.
\]

It should be noted that when $f$ is $e$-convex, under the assumptions of
\cite[Theorem 4.2]{Ali 2022}, $f$ is locally Lipschitz in the interior of its
domain. Hence, $\overline{D}f\left(  x,u\right)  $ is finite for all
$x\in \operatorname*{int}(\operatorname*{dom}f)$ and $u\in X$.

For $e$-convex functions, the following proposition presents an inequality in
terms of the upper Dini directional derivative.

\begin{proposition}
\label{chara} Let $f:X\rightarrow \mathbb{R}\cup \{+\infty \}$ be a proper and
$e$-convex function. Then%
\begin{equation}
\overline{D}f\left(  x,y-x\right)  +\overline{D}f\left(  y,y-x\right)
\leq2e(x,y)\text{ \  \ }\forall x,y\in X. \label{s-Dini}%
\end{equation}

\end{proposition}

\begin{proof}
By assumption $f$ is $e$-convex. Hence, $\operatorname*{dom}f$ is convex. Then
for all $x,y\in$ $\operatorname*{dom}f$ one can have%
\[
\overline{D}f\left(  x,y-x\right)  =\underset{t\rightarrow0^{+}}{\lim \sup
}\frac{f\left(  x+t(y-x)\right)  -f\left(  x\right)  }{t}\leq f\left(
y\right)  -f\left(  x\right)  +e(x,y),
\]
and%
\[
\overline{D}f\left(  y,y-x\right)  =\underset{t\rightarrow0^{+}}{\lim \sup
}\frac{f\left(  y+t(x-y)\right)  -f\left(  y\right)  }{t}\leq f\left(
x\right)  -f\left(  y\right)  +e(x,y).
\]
By adding the above two preceding inequalities, we obtain (\ref{s-Dini}) for
all $x,y\in \operatorname*{dom}f$. Note that if $x\notin \operatorname*{dom}f$,
then $\overline{D}f\left(  x,y-x\right)  =-\infty$, by using the convention
(\ref{Convention})\ we infer (\ref{s-Dini}).
\end{proof}

We recall that (see \cite[page 66]{Zalin. 02}) a function $f:X\rightarrow%
\mathbb{R}
\cup \left \{  \pm \infty \right \}  $ is Lipschitz on a set $A\subset
\operatorname*{dom}f$, if there exists a continuous seminorm $p$ on $X$ such
that%
\begin{equation}
\left \vert f\left(  x\right)  -f\left(  y\right)  \right \vert \leq p\left(
x-y\right)  \  \  \  \forall x,y\in A. \label{Lip-lcs}%
\end{equation}
The considerable size of the class of $e$-convex functions becomes evident
here: indeed, the following proposition shows that every global Lipschitz
function is $e$-convex.Note that Lipschitz functions have inherent differences
from convex functions.

\begin{proposition}
Suppose that $f:X\rightarrow \mathbb{R}\cup \left \{  +\infty \right \}  $ is
Lipschitz on $\operatorname*{dom}$ $f$ with seminorm $p$. Then $f$ is
$e$-convex with%
\[
e\left(  x,y\right)  =2p\left(  x-y\right)  \  \  \  \forall x,y\in
\operatorname*{dom}f.
\]

\end{proposition}

\begin{proof}
Let $x,y\in \operatorname*{dom}f$ and $t\in]0,1[$. By assumption $f$ is
Lipschitz on $\operatorname*{dom}f$. Hence using (\ref{Lip-lcs}) we have%
\begin{align*}
&  f\left(  tx+\left(  1-t\right)  y\right)  -tf\left(  x\right)  -\left(
1-t\right)  f\left(  y\right) \\
&  =t\left(  f\left(  tx+\left(  1-t\right)  y\right)  -f\left(  x\right)
\right)  +\left(  1-t\right)  \left(  f\left(  tx+\left(  1-t\right)
y\right)  -f\left(  y\right)  \right) \\
&  \leq tp\left(  \left(  1-t\right)  \left(  x-y\right)  \right)  +\left(
1-t\right)  p\left(  t\left(  x-y\right)  \right)  =t\left(  1-t\right)
2p\left(  x-y\right)  .
\end{align*}
The proof is complete.
\end{proof}

\section{$e$-conjugacy and its properties}

Let us start this section by introducing and investigating the notion of
$(e,y)$-conjugate in locally convex spaces.

\begin{definition}
Suppose that $f:X\rightarrow \mathbb{R}\cup \{+\infty \}$ is a proper and
$e$--convex function, and $y\in X$ is fixed. Then the map $f_{e,y}^{\ast
}:X^{\ast}\rightarrow \mathbb{R}\cup \{+\infty \}$ defined by%
\[
f_{e,y}^{\ast}\left(  x^{\ast}\right)  :=\sup_{x\in X}\left \{  \left \langle
x^{\ast},x\right \rangle -f\left(  x\right)  -e\left(  x,y\right)  \right \}
,\  \  \  \forall x^{\ast}\in X^{\ast}%
\]
is called the $(e,y)$-conjugate of $f$. Also,\ the function $f_{e,y}^{\ast
\ast}:X\rightarrow \mathbb{R}\cup \{ \pm \infty \}$ defined by%
\[
f_{e,y}^{\ast \ast}\left(  x\right)  :=\sup_{x^{\ast}\in X^{\ast}}\left \{
\left \langle x^{\ast},x\right \rangle -f_{e,y}^{\ast}\left(  x^{\ast}\right)
\right \}  ,\  \  \  \forall x\in X
\]
is called the $(e,y)$-biconjugate of $f$.
\end{definition}

Note that in the definition of $(e,y)$-biconjugate, we do not take
$e$-conjugate twice; this is the (classical convex) conjugate of the
$e$-conjugate. Actually, $f_{e,y}^{\ast \ast}$ is $(f_{e,y}^{\ast})^{\ast}$ and
it is not $\left(  (f_{e,y}^{\ast})_{e,y}\right)  ^{\ast}$. Besides, it should
notice that the following relationship between the $f_{e,y}^{\ast}$ and the
conjugate of $f+e\left(  \cdot,y\right)  $:%
\begin{equation}
f_{e,y}^{\ast}\left(  x^{\ast}\right)  =\sup_{x\in X}\left \{  \left \langle
x^{\ast},x\right \rangle -(f+e\left(  \cdot,y\right)  )\left(  x\right)
\right \}  =(f+e\left(  \cdot,y\right)  )^{\ast}\left(  x^{\ast}\right)
\label{f-ey R f+e}%
\end{equation}
and so%
\begin{equation}
f_{e,y}^{\ast \ast}\left(  x\right)  =(f+e\left(  \cdot,y\right)  )^{\ast \ast
}\left(  x\right)  \leq(f+e\left(  \cdot,y\right)  )\left(  x\right)
=f\left(  x\right)  +e\left(  x,y\right)  . \label{f^**-eyRf+e}%
\end{equation}

\begin{remark}
Suppose that $X$ is a separated locally convex topological vector space, its
topology being defined by a family of seminorms $\mathcal{P}$ such that for
$\delta>0$ and $p_{1},p_{2},...,p_{n}\in \mathcal{P}$,%
\[
B\left(  \delta,p_{1},p_{2},...,p_{n}\right)  :=\left \{  x\in X:\max_{1\leq
i\leq n}p_{i}\left(  x\right)  <\delta \right \}
\]
is a neighborhood base for the origin (see also \cite{Ali 2022, Amin-ali}).
Now, assume that $q:X\times X\rightarrow \mathbb{R}$ is defined by%
\[
q\left(  x,y\right)  :=\max_{i=1,2,...,n}p_{i}\left(  x\right)  +\max
_{i=1,2,...,n}p_{i}\left(  y\right)  .
\]
It is not difficult to check that $q$ is a seminorm on $X\times X$.
\end{remark}

According to the previous remark, we consider the locally Lipschitz property
of the error function as follows.

\emph{\textquotedblleft An error function }$e:X\times X\rightarrow
R+\cup \{+\infty \}$\emph{ is said to be Lipschitz of near }$(x_{0},y_{0}%
)\in \operatorname*{dom}f\times \operatorname*{dom}f$\emph{ if there exists a
neighborhood }$V=B\left(  \delta,p_{1},p_{2},...,p_{n}\right)  $\emph{ of the
origin such that }$\left(  x_{0}+V\right)  \times \left(  y_{0}+V\right)
\subset \operatorname*{dom}f\times \operatorname*{dom}f$\emph{ and for all
}$\left(  x_{1},y_{1}\right)  ,\left(  x_{2},y_{2}\right)  \in \left(
x_{0}+V\right)  \times \left(  y_{0}+V\right)  $%
\begin{align*}
\left \vert e\left(  x_{1},y_{1}\right)  -e\left(  x_{2},y_{2}\right)
\right \vert  &  \leq q\left(  \left(  x_{1},y_{1}\right)  -\left(  x_{2}%
,y_{2}\right)  \right) \\
&  =\left(  \max_{i=1,2,...,n}p_{i}\left(  x_{1}-x_{2}\right)  +\max
_{i=1,2,...,n}p_{i}\left(  y_{1}-y_{2}\right)  \right)
.\text{\textquotedblright}%
\end{align*}
\  \  \  \ The following proposition presents some basic properties of $e$-conjugacy.

\begin{proposition}
\label{properties} Suppose $f,g:X\rightarrow%
\mathbb{R}
\cup \left \{  +\infty \right \}  $ are $e$-convex and $\operatorname*{dom}%
f\cap \operatorname*{dom}g\neq \emptyset$. Then the following statements hold:

(i) $f_{e,y}^{\ast}\left(  x^{\ast}\right)  \leq f^{\ast}(x^{\ast})$. Hence if
$f^{\ast}(\cdot)$ is proper, then so is $f_{e,y}^{\ast}\left(  \cdot \right)  $;

(ii) if $f\leq g$, then $g_{e,y}^{\ast}(x^{\ast})\leq f_{e,y}^{\ast}\left(
x^{\ast}\right)  $;

(iii) if $e\left(  \cdot,y\right)  \leq e^{\prime}\left(  \cdot,y\right)  $,
then $f_{e^{\prime},y}^{\ast}\left(  x^{\ast}\right)  \leq f_{e,y}^{\ast
}\left(  x^{\ast}\right)  $;

(iv) if $h\left(  x\right)  :=f\left(  x\right)  +\lambda$ where $\lambda \in%
\mathbb{R}
$, then $h_{e,y}^{\ast}\left(  x^{\ast}\right)  =f_{e,y}^{\ast}\left(
x^{\ast}\right)  -\lambda$;

(v) if $\lambda>0$ and $h\left(  x\right)  :=\lambda f\left(  x\right)  $,
then $h_{e,y}^{\ast}\left(  x^{\ast}\right)  =\lambda f_{\frac{e}{\lambda}%
,y}^{\ast}\left(  \frac{x^{\ast}}{\lambda}\right)  $;

(vi) if $\lambda>0,$ $e$\ is positively homogeneous\ of degree $k>0$ and
$h\left(  x\right)  :=f\left(  \lambda x\right)  $, then $h_{e,y}^{\ast
}(x^{\ast})=f_{\frac{e}{\lambda^{k}},\lambda y}^{\ast}\left(  \frac{x^{\ast}%
}{\lambda}\right)  $;

(vii) if $e\left(  \cdot,x\right)  \left \vert _{\operatorname*{dom}f}\right.
$ is locally Lipschitz on $\operatorname*{int}(\operatorname*{dom}f)$ for each
$x\in X$, then $f_{e,y}^{\ast}$ is locally Lipschitz on $\operatorname*{int}%
(\operatorname*{dom}f)$ as a function of $y$;

(viii) $e$-conjugacy is convex, i.e., for every $\lambda \in]0,1[$ the
following inequality holds%
\[
\left(  \lambda f+(1-\lambda)g\right)  _{e,y}^{\ast}\left(  x^{\ast}\right)
\leq \lambda f_{e,y}^{\ast}\left(  x^{\ast}\right)  +\left(  1-\lambda \right)
g_{e,y}^{\ast}\left(  x^{\ast}\right)  \text{;}%
\]

(ix) for every $x\in X$, one has $f_{e,y}^{\ast \ast}\left(  x\right)  \leq
f\left(  x\right)  +e\left(  x,y\right)  $;

(x) assume that $e\left(  \cdot,w\right)  \left \vert _{\operatorname*{dom}%
f}\right.  $ is upper semicontinuous for each $w\in \operatorname*{dom}f$ and
the $e$ is defined on $X\times X$ and has the triangle inequality property,
i.e.,%
\begin{equation}
e(z,x)\leq e(z,y)+e(y,x)\text{ \  \ }\forall x,y,z\in X. \label{tri-an}%
\end{equation}
Then%
\begin{equation}
\partial^{CR}f\left(  x\right)  \neq \emptyset \implies f\left(  x\right)
-e\left(  x,y\right)  \leq f_{e,y}^{\ast \ast}\left(  x\right)  \  \  \  \forall
x\in \operatorname*{dom}f. \label{4 in Ay}%
\end{equation}

\end{proposition}

\begin{proof}
The proofs of items (i)--(v) follow easily from the definition.\newline To
show (vi), using the positively homogeneous assumption, by setting $z=\lambda
x$ (for $\lambda>0$) we deduce that%
\begin{align*}
h_{e,y}^{\ast}(x^{\ast})  &  =\sup_{z\in X}\left \{  \left \langle x^{\ast
},\frac{z}{\lambda}\right \rangle -f\left(  z\right)  -e\left(  \frac
{z}{\lambda},y\right)  \right \} \\
&  =\sup_{z\in X}\left \{  \left \langle x^{\ast},\frac{z}{\lambda}\right \rangle
-f\left(  z\right)  -e\left(  \frac{z}{\lambda},\frac{\lambda y}{\lambda
}\right)  \right \} \\
&  =\sup_{z\in X}\left \{  \left \langle \frac{x}{\lambda}^{\ast},z\right \rangle
-f\left(  z\right)  -\frac{e\left(  z,\lambda y\right)  }{\lambda^{k}%
}\right \}  =f_{\frac{e}{\lambda^{k}},\lambda y}^{\ast}\left(  \frac{x^{\ast}%
}{\lambda}\right)  .
\end{align*}
For assertion (vii), fix $x_{0}\in \operatorname*{int}(\operatorname*{dom}f)$
and $x\in X$. Take a neighborhood $V=B\left(  \delta,p_{1},p_{2}%
,...,p_{n}\right)  $ of the origin such that for each$(x,y_{1}),(x,y_{2})\in
X\times(x_{0}+V)$%
\begin{align}
|e(x,y_{2})-e(x,y_{1})|  &  \leq q((x,y_{2})-(x,y_{1}))\label{lip-sch5}\\
&  =q(0,y_{2}-y_{1})=\max_{i=1,2,...,n}p_{i}\left(  y_{2}-y_{1}\right)
\text{.}\nonumber
\end{align}
By the definition of $f_{e,y}^{\ast}$, for all $x^{\ast}\in X^{\ast},x\in
X,y_{1},y_{2}\in x_{0}+V$ and using inequality (\ref{lip-sch5}) we get%
\begin{align*}
f_{e,y_{1}}^{\ast}(x^{\ast})  &  \geq \left \langle x^{\ast},x\right \rangle
-f\left(  x\right)  -e\left(  x,y_{1}\right)  +e\left(  x,y_{2}\right)
-e\left(  x,y_{2}\right) \\
&  \geq \left \langle x^{\ast},x\right \rangle -f\left(  x\right)  -e\left(
x,y_{2}\right)  -\max_{i=1,2,...,n}p_{i}\left(  y_{2}-y_{1}\right)  \text{.}%
\end{align*}
Taking the supremum over $x\in X$ in the above inequality, we have%
\[
f_{e,y_{1}}^{\ast}(x^{\ast})\geq f_{e,y_{2}}^{\ast}(x^{\ast})-\max
_{i=1,2,...,n}p_{i}\left(  y_{2}-y_{1}\right)  \  \  \  \forall y_{1},y_{2}\in
x_{0}+V\text{.}%
\]
Switching the roles of $y_{1}$ and $y_{2}$ in the above inequality allows us
to deduce that%
\[
\left \vert f_{e,y_{1}}^{\ast}(x^{\ast})-f_{e,y_{2}}^{\ast}(x^{\ast
})\right \vert \leq \max_{i=1,2,...,n}p_{i}\left(  y_{2}-y_{1}\right)
\  \  \  \forall y_{1},y_{2}\in x_{0}+V\text{.}%
\]
For the proof of item (viii), using (\ref{f-ey R f+e}) we have%
\begin{align*}
\left(  \lambda f+(1-\lambda)g\right)  _{e,y}^{\ast}\left(  x^{\ast}\right)
&  =\left[  \lambda \left(  f+e\left(  \cdot,y\right)  \right)  +(1-\lambda
)\left(  g+e\left(  \cdot,y\right)  \right)  \right]  ^{\ast}\left(  x^{\ast
}\right) \\
&  \leq \lambda \left(  f+e\left(  \cdot,y\right)  \right)  ^{\ast}\left(
x^{\ast}\right)  +(1-\lambda)\left(  g+e\left(  \cdot,y\right)  \right)
^{\ast}\left(  x^{\ast}\right) \\
&  =\lambda f_{e,y}^{\ast}\left(  x^{\ast}\right)  +\left(  1-\lambda \right)
g_{e,y}^{\ast}\left(  x^{\ast}\right)  .
\end{align*}
\newline The prove (ix), follows from (\ref{f-ey R f+e}) and
(\ref{f^**-eyRf+e}).\newline To show (x), let $y\in X$ and $x\in
\operatorname*{dom}f$. Take any $x^{\ast}\in \partial^{CR}f(x)$. By
\cite[Theorem 3.5]{Ali 2022}, $x^{\ast}\in \partial^{e}f(x)$. One has%
\[
\left \langle x^{\ast},z-x\right \rangle \leq f\left(  z\right)  -f\left(
x\right)  +e\left(  x,z\right)  \  \  \  \forall z\in X\text{.}%
\]
Use the inequality $e(z,x)\leq e(z,y)+e(y,x)$ to deduce%
\[
\left \langle x^{\ast},z\right \rangle -f\left(  z\right)  -e\left(  z,y\right)
\leq \left \langle x^{\ast},x\right \rangle -f\left(  x\right)  +e\left(
x,y\right)  \  \  \  \forall y,z\in X\text{.}%
\]
Taking the supremum from the left-hand side over $z\in X$, we get%

\[
f_{e,y}^{\ast}(x^{\ast})\leq \left \langle x^{\ast},x\right \rangle
-f(x)+e(x,y)\  \  \  \forall y\in X\text{.}%
\]
Consequently, by definition $(e,y)$-biconjugate of $f$, for all $y\in X$%
\[
f(x)-e(x,y)\leq \left \langle x^{\ast},x\right \rangle -f_{e,y}^{\ast}(x^{\ast
})\leq f_{e,y}^{\ast \ast}(x)\text{,}%
\]
and we conclude the desired inequality.
\end{proof}

Next, we characterize the $e$-subdifferential by using the $(e,y)$-conjugacy.
See also \cite{Ali 2021, NLM 97}.

\begin{proposition}
\label{19}Let $f:X\rightarrow%
\mathbb{R}
\cup \left \{  +\infty \right \}  $ be an $e$-convex function such that for each
$x\in \operatorname*{dom}f,$ $e\left(  x,x\right)  =0$. Then for each
$x\in \operatorname{dom}f$, $x^{\ast}\in X^{\ast}$,
\begin{equation}
x^{\ast}\in \partial^{e}f\left(  x\right)  \iff f_{e,x}^{\ast}\left(  x^{\ast
}\right)  +f\left(  x\right)  =\left \langle x^{\ast},x\right \rangle .
\label{Fen-Morrea equality}%
\end{equation}
In particular, $\operatorname*{gr}\left(  \partial^{e}f\right)  \subset
\operatorname*{dom}f\times \operatorname*{dom}f_{e,x}^{\ast}$.
\end{proposition}

\begin{proof}
From the definition of $f_{e,y}^{\ast}$ we obtain $f_{e,y}^{\ast}\left(
x^{\ast}\right)  +f\left(  x\right)  +e\left(  x,y\right)  \geq \left \langle
x^{\ast},x\right \rangle $. Choose $x=y$\ in the previous inequality. Then
$f_{e,x}^{\ast}\left(  x^{\ast}\right)  +f\left(  x\right)  \geq \left \langle
x^{\ast},x\right \rangle $. The reverse inequality follows from the subsequent
equivalent assertions;%
\begin{align*}
x^{\ast}  &  \in \partial^{e}f\left(  x\right)  \iff \left \langle x^{\ast
},y-x\right \rangle \leq f\left(  y\right)  -f\left(  x\right)  +e\left(
x,y\right)  \qquad \forall y\in X\\
&  \iff \left \langle x^{\ast},y\right \rangle -f\left(  y\right)  -e\left(
x,y\right)  \leq \left \langle x^{\ast},x\right \rangle -f\left(  x\right)
\qquad \forall y\in X\\
&  \iff \sup_{y\in X}\left(  \left \langle x^{\ast},y\right \rangle -f\left(
y\right)  -e\left(  x,y\right)  \right)  \leq \left \langle x^{\ast
},x\right \rangle -f\left(  x\right) \\
&  \iff f_{e,x}^{\ast}\left(  x^{\ast}\right)  \leq \left \langle x^{\ast
},x\right \rangle -f\left(  x\right)  .
\end{align*}
The proof of the, in particular, part deduces from (\ref{Fen-Morrea equality}).
\end{proof}

\begin{proposition}
Let $f:X\rightarrow%
\mathbb{R}
\cup \left \{  +\infty \right \}  $ be a proper and $e$-convex function. Assume
that $e$ is defined on $X\times X$, has the triangle inequality property on
$X\times X$, and also for all $z\in \operatorname*{dom}f$, $e\left(
z,z\right)  =0$ and $e\left(  \cdot,z\right)  |_{\operatorname*{dom}f}$ is
upper semicontinuous. If $x\in \operatorname*{dom}f$ and $\partial^{CR}f\left(
x\right)  \neq \emptyset$, then the following assertions are equivalent:

(i) $x^{\ast}\in \partial^{e}f\left(  x\right)  ;$

(ii) $x\in \partial f_{e,x}^{\ast}\left(  x^{\ast}\right)  ;$

(iii) $x^{\ast}\in \partial f_{e,x}^{\ast \ast}\left(  x\right)  .$
\end{proposition}

\begin{proof}
To show (i)$\implies$(ii), let $x^{\ast}\in \partial^{e}f\left(  x\right)  $.
Then from Proposition \ref{19}\ we get%
\begin{equation}
f\left(  x\right)  =\left \langle x^{\ast},x\right \rangle -f_{e,x}^{\ast
}\left(  x^{\ast}\right)  . \label{fen1}%
\end{equation}
Besides, using the definition of $f_{e,x}^{\ast}\left(  z^{\ast}\right)  $
implies that%
\begin{equation}
f_{e,y}^{\ast}\left(  z^{\ast}\right)  \geq \left \langle z^{\ast}%
,x\right \rangle -f\left(  x\right)  -e\left(  x,y\right)  \quad \forall y\in
X,\forall z^{\ast}\in X^{\ast}, \label{s-c1}%
\end{equation}
Put $y=x$ in (\ref{s-c1}). Hence%
\begin{equation}
f\left(  x\right)  \geq \left \langle z^{\ast},x\right \rangle -f_{e,y}^{\ast
}\left(  z^{\ast}\right)  \  \  \  \forall z^{\ast}\in X^{\ast}. \label{s-c2}%
\end{equation}
From (\ref{fen1}) and (\ref{s-c2}), we infer that%
\[
\left \langle z^{\ast}-x^{\ast},x\right \rangle \leq f_{e,x}^{\ast}\left(
z^{\ast}\right)  -f_{e,x}^{\ast}\left(  x^{\ast}\right)  \  \  \  \  \forall
z^{\ast}\in X^{\ast}.
\]
Consequently $x\in \partial f_{e,x}^{\ast}\left(  x^{\ast}\right)  $.\newline
To prove the implication (ii)$\implies$(iii), assume that $x\in \partial
f_{e,x}^{\ast}\left(  x^{\ast}\right)  $. Hence for each $y^{\ast}\in X^{\ast
}$ we obtain $\left \langle y^{\ast},x\right \rangle -f_{e,x}^{\ast}\left(
y^{\ast}\right)  \leq \left \langle x^{\ast},x\right \rangle -f_{e,x}^{\ast
}\left(  x^{\ast}\right)  $. Take the supremum over $y^{\ast}\in X^{\ast}$,
one has $f_{e,x}^{\ast \ast}\left(  x\right)  \leq \left \langle x^{\ast
},x\right \rangle -f_{e,x}^{\ast}\left(  x^{\ast}\right)  $ or $\left \langle
x^{\ast},-x\right \rangle +f_{e,x}^{\ast}\left(  x^{\ast}\right)  \leq
-f_{e,x}^{\ast \ast}\left(  x\right)  $. Thus%
\begin{equation}
\left \langle x^{\ast},y-x\right \rangle -\left \langle x^{\ast},y\right \rangle
+f_{e,x}^{\ast}\left(  x^{\ast}\right)  \leq-f_{e,x}^{\ast \ast}\left(
x\right)  \  \  \  \forall y\in X. \label{s-c3}%
\end{equation}
Using the definition of $f_{e,x}^{\ast \ast}\left(  y\right)  $, we get%
\begin{equation}
-f_{e,x}^{\ast \ast}\left(  y\right)  \leq-\left \langle x^{\ast},y\right \rangle
+f_{e,x}^{\ast}\left(  x^{\ast}\right)  \  \  \  \forall y\in X,\forall x^{\ast
}\in X^{\ast}. \label{s-c4}%
\end{equation}
From (\ref{s-c3}) and (\ref{s-c4}) we deduce that%
\[
\left \langle x^{\ast},y-x\right \rangle -f_{e,x}^{\ast \ast}\left(  y\right)
\leq \left \langle x^{\ast},y-x\right \rangle -\left \langle x^{\ast
},y\right \rangle +f_{e,x}^{\ast}\left(  x^{\ast}\right)  \leq-f_{e,x}%
^{\ast \ast}\left(  x\right)  \  \  \  \forall y\in X.
\]
It follows from the above inequalities%
\[
\left \langle x^{\ast},y-x\right \rangle \leq f_{e,x}^{\ast \ast}\left(
y\right)  -f_{e,x}^{\ast \ast}\left(  x\right)  \  \  \  \forall y\in X.
\]
Therefore $x^{\ast}\in \partial f_{e,x}^{\ast \ast}\left(  x\right)  $.\newline
Now,\ we show (iii)$\implies$(i). Let $x^{\ast}\in \partial f_{e,x}^{\ast \ast
}\left(  x\right)  $. Then%
\begin{equation}
\left \langle x^{\ast},y-x\right \rangle \leq f_{e,x}^{\ast \ast}\left(
y\right)  -f_{e,x}^{\ast \ast}\left(  x\right)  \  \  \  \forall y\in X.
\label{3-2022}%
\end{equation}
Applying the definitions of $f_{e,x}^{\ast \ast}$ and $f_{e,x}^{\ast}$,\ we
deduce that%
\begin{equation}
f_{e,x}^{\ast \ast}\left(  y\right)  \leq f\left(  y\right)  +e\left(
x,y\right)  \  \  \  \forall x\in X. \label{1-2022}%
\end{equation}
Choose $y=x$\ in (\ref{4 in Ay}) and (\ref{1-2022})\ we obtain%
\begin{equation}
f_{e,x}^{\ast \ast}\left(  x\right)  =f\left(  x\right)  . \label{2-2022}%
\end{equation}
Now from (\ref{3-2022}), (\ref{1-2022}) and (\ref{2-2022}) we conclude that%
\[
\left \langle x^{\ast},y-x\right \rangle \leq f\left(  y\right)  -f\left(
x\right)  +e\left(  x,y\right)  \  \  \  \forall y\in X,
\]
i.e., $x^{\ast}\in \partial^{e}f\left(  x\right)  $.
\end{proof}

\begin{proposition}
Suppose that $e$ is defined on $X\times X$ and $e$ has the triangle inequality
property on $X\times X$. Then%
\[
\left \vert f_{e,y_{1}}^{\ast}\left(  x^{\ast}\right)  -f_{e,y_{2}}^{\ast
}\left(  x^{\ast}\right)  \right \vert \leq e\left(  y_{1},y_{2}\right)
\  \  \  \forall y_{1},y_{2}\in X,x^{\ast}\in X^{\ast}.
\]

\end{proposition}

\begin{proof}
Assume that $y_{1},y_{2}\in X$ and $x^{\ast}\in X^{\ast}$. Then using triangle
inequality property on $e$ we get%
\begin{align*}
f_{e,y_{1}}^{\ast}\left(  x^{\ast}\right)   &  \leq \sup_{x\in X}\left \{
\left \langle x^{\ast},x\right \rangle -f\left(  x\right)  -e\left(
x,y_{2}\right)  +e\left(  y_{2},y_{1}\right)  \right \} \\
&  =\sup_{x\in X}\left \{  \left \langle x^{\ast},x\right \rangle -f\left(
x\right)  -e\left(  x,y_{2}\right)  \right \}  +e\left(  y_{2},y_{1}\right)
=f_{e,y_{2}}^{\ast}\left(  x^{\ast}\right)  +e\left(  y_{2},y_{1}\right)  .
\end{align*}
By switching the roles of $y_{1}$ and $y_{2}$ in the above inequality, we
deduce the desired statement.
\end{proof}

In the sequel, we will provide a nontrivial example of an error function $e$,
such that $e(\cdot,y)$ is bounded from above for each $y\in X$.

\begin{example}
Define $f:%
\mathbb{R}
\rightarrow%
\mathbb{R}
$ by $f\left(  x\right)  =x\exp \left(  -x\right)  $. Then $f$ is $e$-convex
with the error function $e\left(  x,y\right)  =\left(  \exp \left(  -x\right)
-\exp \left(  -y\right)  \right)  \left(  y-x\right)  $. It is not hard to
observed that $e(\cdot,y)$ is bounded from above for each $y\in%
\mathbb{R}
$.
\end{example}

\begin{proof}
Let $x,y\in%
\mathbb{R}
$ and $t\in]0,1[$. We have%
\begin{align*}
&  f\left(  tx+\left(  1-t\right)  y\right)  -tf\left(  x\right)  -\left(
1-t\right)  f\left(  y\right) \\
&  =\left(  tx+\left(  1-t\right)  y\right)  \exp \left(  -\left(  tx+\left(
1-t\right)  y\right)  \right)  -tx\exp \left(  -x\right)  -\left(  1-t\right)
y\exp \left(  -y\right) \\
&  =tx\left[  \exp \left(  -tx-\left(  1-t\right)  y\right)  -\exp \left(
-x\right)  \right]  +\left(  1-t\right)  y\left[  \exp \left(  -tx-\left(
1-t\right)  y\right)  -\exp \left(  -y\right)  \right] \\
&  \leq tx\left[  t\exp \left(  -x\right)  +\left(  1-t\right)  \exp \left(
-y\right)  -\exp \left(  -x\right)  \right] \\
&  \  \  \  \  \  \  \  \  \  \  \  \  \  \  \  \  \  \  \  \ +\left(  1-t\right)  y\left[
t\exp \left(  -x\right)  +\left(  1-t\right)  \exp \left(  -y\right)
-\exp \left(  -y\right)  \right] \\
&  =tx\left[  \left(  1-t\right)  \exp \left(  -y\right)  -\left(  1-t\right)
\exp \left(  -x\right)  \right]  +\left(  1-t\right)  y\left[  t\exp \left(
-x\right)  -t\exp \left(  -y\right)  \right] \\
&  =t\left(  1-t\right)  \left(  \exp \left(  -x\right)  -\exp \left(
-y\right)  \right)  \left(  y-x\right)  .
\end{align*}
The above inequality follows from the convexity of function $g\left(
x\right)  =\exp \left(  -x\right)  $.
\end{proof}

\begin{proposition}
Let $e$ defined on $X\times X$ and $e(\cdot,y)$ bounded from above for each
$y\in X$ and let $k_{y}:=\sup_{x\in X}e\left(  x,y\right)  $. Then%
\[
\left \vert f_{e,y_{1}}^{\ast}\left(  x^{\ast}\right)  -f_{e,y_{2}}^{\ast
}\left(  x^{\ast}\right)  \right \vert \leq k_{y_{1},y_{2}}\  \  \  \forall
y_{1},y_{2}\in X,x^{\ast}\in X^{\ast},
\]
where $k_{y_{1},y_{2}}=\max \left \{  k_{y_{1}},k_{2}\right \}  $
\end{proposition}

\begin{proof}
Fix $y_{1},y_{2}\in X$ and $x^{\ast}\in X^{\ast}$. Then%
\begin{align*}
f_{e,y_{1}}^{\ast}\left(  x^{\ast}\right)   &  =\sup_{x\in X}\left \{
\left \langle x^{\ast},x\right \rangle -f\left(  x\right)  -e\left(
x,y_{1}\right)  +e\left(  x,y_{2}\right)  -e\left(  x,y_{2}\right)  \right \}
\\
&  \leq \sup_{x\in X}\left \{  \left \langle x^{\ast},x\right \rangle -f\left(
x\right)  -e\left(  x,y_{2}\right)  \right \}  +\sup_{x\in X}\left \{  e\left(
x,y_{2}\right)  -e\left(  x,y_{1}\right)  \right \} \\
&  \leq f_{e,y2}^{\ast}\left(  x^{\ast}\right)  +\sup_{x\in X}e\left(
x,y_{2}\right)  -\inf_{x\in X}e\left(  x,y_{1}\right) \\
&  \leq f_{e,y_{2}}^{\ast}\left(  x^{\ast}\right)  +k_{y_{2}}\leq f_{e,y_{2}%
}^{\ast}\left(  x^{\ast}\right)  +k_{y_{1},y_{2}}%
\end{align*}
The third inequality holds because $\inf_{x\in X}e\left(  x,y_{1}\right)
\geq0$. Switching the roles of $y_{1}$ and $y_{2}$ in the above inequality
allows us to obtain that the desired statement.
\end{proof}

The following result under some assumptions shows that the $(e,y)$-conjugate
and $(e,y)$-biconjugate functions are proper. In addition, it provides an
affine minorant for $f(\cdot)+e(\cdot,y)$.

\begin{theorem}
Let an error function $e$ be defined on $X\times X$ and bounded from above.
Suppose that $f:X\rightarrow \mathbb{R}\cup \{+\infty \}$ is proper and
$e(x,x)=0$ for each $x\in \operatorname*{dom}f$, also for $x_{0}\in
\operatorname*{dom}f$, $\operatorname*{int}\left(  \operatorname*{co}\left(
\operatorname*{epi}\left(  f,e,x_{0}\right)  \right)  \right)  $ is nonempty
and $f(\cdot)+e(\cdot,x_{0})$ is bounded above on a neighborhood of $x_{0}$.
If $\left(  x_{0},f\left(  x_{0}\right)  \right)  \in \operatorname*{bd}\left(
\operatorname*{co}\left(  \operatorname*{epi}\left(  f,e,x_{0}\right)
\right)  \right)  $, then for every $y\in X$, $f_{e,y}^{\ast}$, and
$f_{e,y}^{\ast \ast}$ are proper. Moreover, $f(x)+e(x,y)$ is minorized by an
affine map.
\end{theorem}

\begin{proof}
According to \cite[Proposition 3.10]{Ali 2022}, we infer that $\partial
^{e}f(x_{0})$ is nonempty, so let $x^{\ast}\in \partial^{e}f(x_{0})$. Therefore
for all $z\in X$, we have $\left \langle x^{\ast},z-x_{0}\right \rangle \leq
f\left(  z\right)  -f\left(  x_{0}\right)  +e\left(  z,x_{0}\right)  $. From
this inequality, we get%
\begin{equation}
\left \langle x^{\ast},z\right \rangle -f\left(  z\right)  -e\left(  y,z\right)
\leq \left \langle x^{\ast},x_{0}\right \rangle -f\left(  x_{0}\right)  +e\left(
z,x_{0}\right)  -e\left(  y,z\right)  . \label{AY21}%
\end{equation}
Since $e$ is bounded from above on $X\times X$, thus there exists a real
number $M$ such that%
\begin{equation}
e\left(  x,z\right)  \leq M\  \  \  \forall \left(  x,z\right)  \in X\times X.
\label{AY22}%
\end{equation}
According to the inequalities (\ref{AY21}), (\ref{AY22}), and nonnegativity of
$e$ for all $z\in X$ we have%
\[
\left \langle x^{\ast},z\right \rangle -f\left(  z\right)  -e\left(  y,z\right)
\leq \left \langle x^{\ast},x_{0}\right \rangle -f\left(  x_{0}\right)  +M.
\]
From the left-hand side, by taking the supremum over $z\in X$, we obtain
$f_{e,y}^{\ast}(x^{\ast})<+\infty$. Therefore we get $f_{e,y}^{\ast}$ is
proper. Since $f_{e,y}^{\ast}(x^{\ast})\in \mathbb{R}$ for at least one
$x^{\ast}\in X^{\ast}$, the definition of $f_{e,y}^{\ast \ast}$ implies that
$f_{e,y}^{\ast \ast}(z)>-\infty$ for all $z\in X$. Ultimately, Proposition
\ref{properties}-(ix) shows that $f_{e,y}^{\ast \ast}$ is also proper. For
proof of the last part of the theorem, since $f_{e,y}^{\ast \ast}$ is convex,
lower semicontinuous and proper. Thus it is minorized by an affine map. Again
by using Proposition \ref{properties}-(ix), we conclude the desired result.
\end{proof}

\section{Applications in optimization}

In this section, we attempt to recover some results of optimization, similar
to Fermat's rule, in terms of the $e$-subdifferential.

The following proposition is the generalization of the result from
\cite[Theorem 4.5]{Hui-Sun} to the $e$--convex functions. Although, their
proof works for the $e$-convex functions with some modifications and
adjustments, we include proof just for the sake of completeness.

\begin{proposition}
Let $f,g:$ $X\rightarrow%
\mathbb{R}
\cup \left \{  +\infty \right \}  $ be proper and $e$-convex functions, and
$x_{0}\in \operatorname*{dom}f\cap \operatorname*{dom}g$. If $f-g$ attains its
global minimum at $x_{0}$, then $\partial^{e}g\left(  x_{0}\right)
\subset \partial^{e}f\left(  x_{0}\right)  $.
\end{proposition}

\begin{proof}
By assumption $f-g$ attains its global minimum at $x_{0}$. Therefore%
\begin{equation}
f\left(  x_{0}\right)  -g\left(  x_{0}\right)  \leq f\left(  x\right)
-g\left(  x\right)  \  \  \  \forall x\in X\text{.} \label{Ferm-dif}%
\end{equation}
Suppose that $x^{\ast}\in \partial^{e}g(x_{0})$. By the definition of
$e$-subdifferential and (\ref{Ferm-dif}), we infer that%
\[
\left \langle x^{\ast},x-x_{0}\right \rangle \leq g\left(  x\right)  -g\left(
x_{0}\right)  +e\left(  x,x_{0}\right)  \leq f\left(  x\right)  -f\left(
x_{0}\right)  +e\left(  x,x_{0}\right)  \  \  \  \forall x\in X\text{.}%
\]
Again, by using the definition of $e$-subdifferential, we conclude that
$x^{\ast}\in \partial^{e}f\left(  x_{0}\right)  $.
\end{proof}

\begin{corollary}
Let $f:$ $X\rightarrow%
\mathbb{R}
\cup \left \{  +\infty \right \}  $ be a proper and $e$-convex function, and
$x_{0}\in \operatorname*{dom}f$. If $f$ attains its global minimum at $x_{0}$,
then $0^{\ast}\in \partial^{e}f\left(  x_{0}\right)  $.
\end{corollary}

\begin{proof}
Just take $g\equiv0$ in the above proposition.
\end{proof}

\begin{theorem}
Suppose that $f,g:$ $X\rightarrow%
\mathbb{R}
\cup \left \{  +\infty \right \}  $ are two proper and $e$--convex functions, and
$x_{0}\in \operatorname*{dom}f\cap \operatorname*{dom}g$. Assume that $f-g$
attains its local minimum at $x_{0}$, $e(\cdot,x_{0})$ is convex and
$e(x,x)=0$ for all $x\in \operatorname*{dom}f\cap \operatorname*{dom}g$. Then
$\partial^{e}g\left(  x_{0}\right)  \subset \partial^{2e}f\left(  x_{0}\right)
$.
\end{theorem}

\begin{proof}
Since $f-g$ attains a local minimum at $x_{0}$, there exists $U\in
\mathcal{N}_{x_{0}}$ such that%
\begin{equation}
f\left(  x_{0}\right)  -g\left(  x_{0}\right)  \leq f\left(  x\right)
-g\left(  x\right)  \  \  \  \forall x\in U\text{.} \label{L-Ferma}%
\end{equation}
Fix arbitrary $x\in X$ and $x^{\ast}\in \partial^{e}g(x_{0})$. Set
$x_{n}:=\frac{1}{n}x+\left(  1-\frac{1}{n}\right)  x_{0}$ where $n\in
\mathbb{N}$. So for large enough $n$, we have $x_{n}\in U$. It follows from
the definition of $e$-subdifferential, (\ref{L-Ferma}) and the $e$-convexity
of $f$ and $e\left(  \cdot,x_{0}\right)  $,
\begin{align*}
\left \langle x^{\ast},\frac{1}{n}\left(  x-x_{0}\right)  \right \rangle  &
=\left \langle x^{\ast},x_{n}-x_{0}\right \rangle \\
&  \leq g\left(  x_{n}\right)  -g\left(  x_{0}\right)  +e\left(  x_{n}%
,x_{0}\right) \\
&  \leq f\left(  x_{n}\right)  -f\left(  x_{0}\right)  +e\left(  x_{n}%
,x_{0}\right) \\
&  \leq \frac{1}{n}\left[  f\left(  x\right)  -f\left(  x_{0}\right)  \right]
+\frac{1}{n}\left(  1-\frac{1}{n}\right)  e\left(  x,x_{0}\right)  +\frac
{1}{n}e\left(  x,x_{0}\right)  .
\end{align*}
Multiply both sides of the above inequality by $n$, we get%
\begin{equation}
\left \langle x^{\ast},x-x_{0}\right \rangle \leq f\left(  x\right)  -f\left(
x_{0}\right)  +\left(  1-\frac{1}{n}\right)  e\left(  x,x_{0}\right)
+e\left(  x_{n},x_{0}\right)  . \label{AY-17}%
\end{equation}
Consequently, by tending $n\rightarrow+\infty$ in the above inequality%
\[
\left \langle x^{\ast},x-x_{0}\right \rangle \leq f\left(  x\right)  -f\left(
x_{0}\right)  +2e(x,x_{0}).
\]
Because $x\in X$ is arbitrary, we conclude that $x^{\ast}\in \partial
^{2e}f\left(  x_{0}\right)  $.
\end{proof}

\begin{corollary}
Let $f:$ $X\rightarrow%
\mathbb{R}
\cup \left \{  +\infty \right \}  $ be a proper and $e$-convex function, and
$x_{0}\in \operatorname*{dom}f$. If $f$ attains its local minimum at $x_{0}$,
then $0^{\ast}\in \partial^{2e}f\left(  x_{0}\right)  $.
\end{corollary}

\begin{proof}
Just take $g\equiv0$ in the above theorem.
\end{proof}

In the next result, for the finite-dimensional cases and under some
assumptions, we show that an $e$-convex function has a Fr\'{e}chet derivative
if and only if it has a Gateaux derivative.

\begin{proposition}
Let $f:%
\mathbb{R}
^{n}\rightarrow \mathbb{R}\cup \{+\infty \}$ be a proper and $e$-convex function,
and $f$ is locally bounded above on the $\operatorname*{int}%
(\operatorname*{dom}f)$. If $e|_{\operatorname*{dom}f\times \operatorname*{dom}%
f}$ is locally bounded and $x_{0}\in \operatorname*{int}\left(
\operatorname*{dom}f\right)  $, then the following statements are equivalent:

(i) $f$ is Fr\'{e}chet differentiable at $x_{0}$;

(ii) $f$ is G\^{a}teaux differentiable at $x_{0}$.
\end{proposition}

\begin{proof}
Since all the conditions of \cite[Theorem 4.2]{Ali 2022}\ hold. Thus $f$ is
locally Lipschitz on the interior of its domain, and especially it is locally
Lipschitz at $x_{0}$. Now it follows from \cite[Proposition 3.2]{Mordu-Nam}
that the statements (i) and (ii) are equivalent.
\end{proof}

We recall that if $f,g:X\rightarrow%
\mathbb{R}
\cup \left \{  \pm \infty \right \}  $ are two functions, then their infimal
convolution is a function from $X$ to $%
\mathbb{R}
\cup \left \{  \pm \infty \right \}  $\ defined by%
\[
\left(  f\square g\right)  \left(  x\right)  :=\inf \left \{  f\left(
x_{1}\right)  +g\left(  x_{2}\right)  :x=x_{1}+x_{2}\right \}  .
\]
In the following theorem, under some assumptions, we show that the $\left(
e,y\right)  $-conjugate of the sum of two functions equals the infimal
convolution of their $\left(  e,y\right)  $-conjugates (see also \cite{NLM 97}).

\begin{theorem}
Let $f:X\rightarrow \mathbb{R}\cup \{+\infty \}$ be $e$-convex such that
$e(x,x)=0$ for all $x\in \operatorname*{dom}f$ and $g:X\rightarrow
\mathbb{R}\cup \{+\infty \}$ be $e^{\prime}$-convex such that $e^{\prime
}(z,z)=0$ for each $z\in \operatorname*{dom}g$. Suppose that $f$ and $g$ are
locally bounded above on $\operatorname*{int}\left(  \operatorname*{dom}%
f\right)  $ and $\operatorname*{int}\left(  \operatorname*{dom}g\right)  $,
respectively. Assume that $e(\cdot,x)|_{\operatorname*{dom}f}$ and $e^{\prime
}(\cdot,z)|_{\operatorname*{dom}g}$ are upper semicontinuous for all
$x\in \operatorname*{dom}f$ and $z\in \operatorname*{dom}g$. Then

(i) for every $y\in X$ and $x^{\ast}\in X^{\ast}$ one has%
\begin{equation}
\left(  f+g\right)  _{e+e^{\prime},y}^{\ast}\left(  x^{\ast}\right)
\leq \left(  f_{e,y}^{\ast}\left(  \cdot \right)  \square g_{e^{\prime},y}%
^{\ast}\left(  \cdot \right)  \right)  \left(  x^{\ast}\right)  ;
\label{inf-conv}%
\end{equation}

(ii) if $e|_{\operatorname*{dom}f\times \operatorname*{dom}f}$\ and $e^{\prime
}|_{\operatorname*{dom}g\times \operatorname*{dom}g}$ are locally bounded,
$y\in \operatorname*{int}\left(  \operatorname*{dom}f\right)  \cap
\operatorname*{int}\left(  \operatorname*{dom}g\right)  $ and it is a local
minimum point of the function $f+g-\left \langle x^{\ast},\cdot \right \rangle $,
then the equality holds.
\end{theorem}

\begin{proof}
To prove (i), let $x^{\ast}\in X^{\ast}$ and $x\in X$. Choose $x_{1}^{\ast
},x_{2}^{\ast}\in X^{\ast}$ such that $x_{1}^{\ast}+x_{2}^{\ast}=x^{\ast}$.
Then by the definition of $(e,y)$-conjugate, we have%
\[
f_{e,y}^{\ast}\left(  x_{1}^{\ast}\right)  \geq \left \langle x_{1}^{\ast
},x\right \rangle -f\left(  x\right)  -e\left(  x,y\right)  ,
\]
and%
\[
g_{e^{\prime},y}^{\ast}\left(  x_{2}^{\ast}\right)  \geq \left \langle
x_{2}^{\ast},x\right \rangle -g\left(  x\right)  -e^{\prime}\left(  x,y\right)
.
\]
Adding the above inequalities, we find%
\[
f_{e,y}^{\ast}\left(  x_{1}^{\ast}\right)  +g_{e^{\prime},y}^{\ast}\left(
x_{2}^{\ast}\right)  \geq \left \langle x^{\ast},x\right \rangle -\left(
f+g\right)  \left(  x\right)  -\left(  e+e^{\prime}\right)  \left(
x,y\right)  .
\]
By taking the infimum of the left-hand side over all pairs $(x_{1}^{\ast
},x_{2}^{\ast})$ such that $x_{1}^{\ast}+x_{2}^{\ast}=x^{\ast}$ and by taking
the supremum of the right-hand side over $x$ and, we conclude (\ref{inf-conv}%
).\newline To show (ii), since by assumption $y\in \operatorname*{int}\left(
\operatorname*{dom}f\right)  $ is a local minimum point of the function
$f+g-\left \langle x^{\ast},\cdot \right \rangle $. So using \cite[Theorem
4.2]{Ali 2022}, we infer that $f$ and $g$ are locally Lipschitz on the
interior of their domains. Therefore using the sum rule for the
Clark-Rockafellar subdifferentials \cite[Proposition 2.3.3]{Clarke}, implies
that%
\[
0^{\ast}\in \partial^{CR}\left(  f+g-\left \langle x^{\ast},\cdot \right \rangle
\right)  \left(  y\right)  \subset \partial^{CR}f\left(  y\right)
+\partial^{CR}g\left(  y\right)  -x^{\ast}\text{.}%
\]
Applying \cite[Theorem 3.5]{Ali 2022}, we obtain%
\[
x^{\ast}\in \partial^{CR}f\left(  y\right)  +\partial^{CR}g\left(  y\right)
\subset \partial^{e}f\left(  y\right)  +\partial^{e^{\prime}}g\left(  y\right)
.
\]
Thus there exist $x_{1}^{\ast}\in \partial^{CR}f(y)\subseteq$ $\partial
^{e}f(y)$ and $x_{2}^{\ast}\in \partial^{CR}g(y)\subseteq$ $\partial
^{e^{\prime}}g(y)$ such that $x^{\ast}=x_{1}^{\ast}+x_{2}^{\ast}$. According
to Proposition \ref{19}, we deduce that%
\[
f_{e,y}^{\ast}\left(  x_{1}^{\ast}\right)  =\left \langle x_{1}^{\ast
},y\right \rangle -f\left(  y\right)  ,\qquad g_{e^{\prime},y}^{\ast}\left(
x_{2}^{\ast}\right)  =\left \langle x_{2}^{\ast},y\right \rangle -g\left(
y\right)
\]
and since $f+g$ is an $(e+e^{\prime})$-convex \cite[Proposition 2.4(d)]{Ali
2022}, so using again Proposition \ref{19}, we get%
\[
\left(  f+g\right)  _{e+e^{\prime},y}^{\ast}\left(  x^{\ast}\right)
=\left \langle x^{\ast},y\right \rangle -f\left(  y\right)  -g\left(  y\right)
.
\]
Hence,%
\[
\left(  f_{e,y}^{\ast}\left(  \cdot \right)  \square g_{e^{\prime},y}^{\ast
}\left(  \cdot \right)  \right)  \left(  x^{\ast}\right)  \leq \left(
f+g\right)  _{e+e^{\prime},y}^{\ast}\left(  x^{\ast}\right)  .
\]
It follows that $\left(  f+g\right)  _{e+e^{\prime},y}^{\ast}\left(  x^{\ast
}\right)  =\left(  f_{e,y}^{\ast}\left(  \cdot \right)  \square g_{e^{\prime
},y}^{\ast}\left(  \cdot \right)  \right)  \left(  x^{\ast}\right)  $.
\end{proof}

\bigskip

\textbf{Acknowledgments:} The authors express their gratitude to Professor
Zalinescu for his many insightful comments on the first version of the manuscript.

\bigskip

\textbf{Data Availability: }Data sharing is not applicable to this article as
no datasets were generated or analyzed during the current study.

\bigskip

\textbf{Conflict of interest: }The authors declare that they have no conflict
of interest.

\bigskip

\textbf{Author's contribution}: The author's contributions to the manuscript
are equal.

\end{document}